\documentclass[11pt,a4paper,reqno]{article}
\usepackage[latin1]{inputenc}
\usepackage[T1]{fontenc}
\usepackage[english]{babel}
\usepackage{graphicx}
\usepackage{amsfonts,amssymb,amsbsy,latexsym,a4wide,xspace}
\usepackage{amsmath,delarray,enumerate,calc,amsthm}
\usepackage{bbm}
\usepackage{mathrsfs}
\usepackage{txfonts}
\usepackage{paralist}

\usepackage{authblk}

\usepackage{tikz}
\usetikzlibrary{matrix}

\usepackage{url}

\newcommand\blfootnote[1]{%
  \begingroup
  \renewcommand\thefootnote{}\footnote{#1}%
  \addtocounter{footnote}{-1}%
  \endgroup
}

\newcommand{\delete}[1]{}
\newcommand{\red}[1]{#1}

\newcommand{\supp}{\operatorname{supp}}

\newcommand{\cA}{{\mathcal A}}
\newcommand{\cC}{{\mathcal C}}
\newcommand{\cB}{{\mathcal B}}

\newcommand{\cF}{{\mathcal F}}

\newcommand{\cP}{{\mathcal P}}

\newcommand{\cM}{{\mathcal M}}

\newcommand{\N}{{\mathbbm N}}

\newcommand{\Z}{{\mathbbm Z}}

\newcommand{\1}{{\mathbbm 1}}

\newcommand{\lcm}{\operatorname{lcm}}

\newcommand{\card}{\operatorname{card}}

\newcommand{\ddelta}{\boldsymbol{\delta}}
\newcommand{\oddelta}{\overline{\boldsymbol{\delta}}}
\newcommand{\uddelta}{\underline{\boldsymbol{\delta}}}
\newcommand{\dl}{\underline{d}}
\newcommand{\du}{\overline{d}}

\newcommand{\Spec}{\operatorname{Spec}}

\newtheorem {lemma}{Lemma}
\newtheorem{theorem}{Theorem}
\newtheorem {bemerkung}{Remark}
\newtheorem{proposition}{Proposition}
\newtheorem {corollary}{Corollary}
\newtheorem{beispiel}{Example}
\newtheorem{frage}{Question}
\newtheorem{vermutung}{Conjecture}

\newenvironment{remark} {\begin{bemerkung} \normalfont }{\end{bemerkung}}

\begin{document}


\title{Tautness for sets of multiples and applications to $\cB$-free dynamics}
\author{Gerhard Keller}
\affil{Department of Mathematics, University of Erlangen-N\"urnberg, \\
 Cauerstr. 11, 91058 Erlangen, Germany
 \par Email: keller@math.fau.de}
\date{Version of \today}


\maketitle

{\begin{abstract}
For any set $\cB\subseteq\N=\{1,2,\dots\}$ one can define its \emph{set of multiples}
$\cM_\cB:=\bigcup_{b\in\cB}b\Z$
and the set of \emph{$\cB$-free numbers}
$\cF_\cB:=\Z\setminus\cM_\cB$.
Tautness of the set $\cB$ is a basic property related to questions around the asymptotic density of $\cM_\cB\subseteq\Z$. 
From a dynamical systems point of view (originated in~\cite{Sa}) one studies $\eta$, the indicator function of $\cF_\cB\subseteq\Z$, its shift-orbit closure $X_\eta\subseteq\{0,1\}^\Z$ and the stationary probability measure $\nu_\eta$ defined on $X_\eta$ by the frequencies of finite blocks in $\eta$. In this paper we prove that
tautness implies the following two properties of $\eta$:
\begin{compactenum}[-]
\item The measure $\nu_\eta$ has full topological support in $X_\eta$.
\item If $X_\eta$ is proximal, i.e.~if the one-point set $\{\dots000\dots\}$ is contained in $X_\eta$ and is the unique minimal subset of $X_\eta$, then $X_\eta$ is hereditary, i.e.~if $x\in X_\eta$ and if $w$ is an arbitrary element of $\{0,1\}^\Z$, then also the coordinate-wise product $w\cdot x$ belongs to $X_\eta$.
\end{compactenum}
This strengthens two results from \cite{BKKL2015} which need the stronger assumption that $\cB$ has light tails for the same conclusions.
\end{abstract}}
\blfootnote{\emph{MSC 2010 classification:} 37A45, 37B05, 11B05.}
\blfootnote{\emph{Keywords:} $\cB$-free dynamics, sets of multiples, tautness, $0$-$1$-law.}

\section{Introduction and results}\label{sec:introduction}
For any given set $\cB\subseteq\N=\{1,2,\dots\}$ one can define its \emph{set of multiples}
\begin{equation*}
\cM_\cB:=\bigcup_{b\in\cB}b\Z
\end{equation*}
and the set of \emph{$\cB$-free numbers}
\begin{equation*}
\cF_\cB:=\Z\setminus\cM_\cB\ .
\end{equation*}
The investigation of structural properties of $\cM_\cB$ or, equivalently, of $\cF_\cB$ has a long history (see the monograph \cite{hall-book} and the recent paper \cite{BKKL2015} for references).
Properties of $\cB$ are closely related to properties of the shift dynamical system generated by the two-sided sequence
$\eta\in\{0,1\}^\Z$, the characteristic function of $\cF_\cB$. Indeed, topological dynamics and ergodic theory provide a wealth of concepts to describe various aspects of the structure of $\eta$, see \cite{Sa} which originated
 this point of view by studying the set of square-free numbers, and also
\cite{Peckner2012}, \cite{Ab-Le-Ru}, \cite{BKKL2015}, \cite{KKL2016} for later contributions.

\subsection{A new characterization of tautness} 
In this note we always assume that $\cB$ is \emph{primitive}, i.e. that there are no $b,b'\in\cB$ with $b\mid b'$. 
We recall some  notions from the theory of sets of multiples \cite{hall-book}
and also from \cite{KKL2016}.
\begin{itemize}
\item For a set of multiples $\cM_\cB$ denote by 
\begin{equation*}
\dl(\cM_\cB):=\liminf_{n\to\infty}\frac{1}{n}\sum_{k=1}^n\1_{\cM_\cB}(k)
\;\text{ and }\;
\du(\cM_\cB):=\limsup_{n\to\infty}\frac{1}{n}\sum_{k=1}^n\1_{\cM_\cB}(k)\end{equation*}
the lower and upper density, respectively, and by
\begin{equation*}
\ddelta(\cM_\cB):=\lim_{n\to\infty}\frac{1}{\log n}\sum_{k=1}^nk^{-1}\1_{\cM_\cB}(k)
\end{equation*}
the logarithmic density. Davenport and Erd\"os \cite{DE1936,DE1951} showed that the logarithmic density always exists,  that $\ddelta(\cM_\cB)=\dl(\cM_\cB)$.
\item The set $\cB\subseteq\N\setminus\{1\}$ is a \emph{Behrend set}, if $\ddelta(\cM_\cB)=1$ (in which case also $d(\cM_\cB)=1$).
\item The set $\cB$ is \emph{taut}, if 
\begin{equation*}
\ddelta(\cM_{\cB\setminus\{b\}})<\ddelta(\cM_\cB)\text{\; for each }b\in\cB.
\end{equation*}
So a set is taut, if removing any single point from it changes its set of multiples drastically and not only by ``a few points''. 
\item It is known \cite{hall-book} that $\cB$ is not taut if and only if it  contains a scaled copy of a Behrend set, i.e. if there are $r\in\N$ and a Behrend set $\cA$ such that $r\cA\subseteq\cB$.
\end{itemize}
The logarithmic density of sets of multiples has the following continuity property from below, which is a by-product of the proof of the Davenport-Erd\"os theorem \cite{DE1936,DE1951} (see also \cite{hall-book}):
\begin{equation}\label{eq:app-from-inside}
\ddelta(\cM_\cB)=\lim_{K\to\infty}d(\cM_{\cB\cap\{1,\dots,K\}}).
\end{equation}
At a first glance this property may seem rather close to the following one
\begin{equation}
\lim_{K\to\infty}\overline{d}\left(\cM_{\{b\in\cB:b>K\}}\right)=0,
\end{equation}
which was introduced in \cite{BKKL2015} under the name
\emph{light tails} in order to prove two subtle dynamical properties of the dynamical system associated in a natural way to the set $\cB$ - see the next section for details. 
However it turns out that light tails is definitively a stronger property than \eqref{eq:app-from-inside}. Indeed,
the authors of \cite{BKKL2015} show that each set $\cB$ with light tails is actually taut
and satisfies $\underline{d}(\cB)=\overline{d}(\cB)$, but that the converse does not hold~\cite[Thm.~4.20]{BKKL2015}. They conjecture that tautness might be a sufficient assumption to prove the two dynamical properties alluded to above. In this note we will show that this is indeed the case. A key ingredient to our proof is an apparently new equivalent characterization of Behrend sets in terms of a dichotomy:
\begin{theorem}\label{theo:behrend}
Let $\cB\subseteq\N$ be primitive and denote $\widetilde\cB^{(N)}:=\{b\in\cB:\Spec(b)\cap\{1,\dots,N\}=\emptyset\}$.
\begin{compactenum}[(i)]
\item $\cB$ is Behrend if and only if $\widetilde\cB^{(N)}$ is Behrend 
(i.e. $\ddelta(\cM_{\widetilde\cB^{(N)}})=1$) for all $N\in\N$.
\item $\cB$ is  not Behrend if and only if
\begin{equation}
\lim_{N\to\infty}\ddelta(\cM_{\widetilde{\cB}^{(N)}})=0.
\end{equation}
\end{compactenum}
\end{theorem}
\noindent
The proof, which we present in  section~\ref{sec:behrend}, relies on a version of Kolmogorov's $0$-$1$-law, that is behind Lemma~\ref{lemma:dichotomy} below. Stanis\l{}aw Kasjan found a purely number theoretic proof of this lemma  and was so kind to allow a reproduction of his proof in this paper~\cite{Kasjan-pc}.

A rather immediate corollary to this theorem characterizes taut sets. We use the following notation: For a primitive set $\cB\subseteq\N$ and any positive integer $q$ let
\begin{equation*}
\cB/q:=\{b/q: b\in\cB\text{ and }q\mid  b\}.
\end{equation*}  
\begin{corollary}\label{coro:taut}
A primitive set $\cB\subseteq\N$ is taut if and only if
$\lim_{N\to\infty}\ddelta(\cM_{\widetilde{\cB/q}^{(N)}})=0$ for all $q\in\N\setminus\cB$.
\end{corollary}
\begin{proof}
Suppose first that $\cB$ is taut. As $q\cdot\cB/q\subseteq\cB$ and $1\not\in\cB/q$, the set $\cB/q$ is not Behrend, and Theorem~\ref{theo:behrend} implies $\lim_{N\to\infty}\ddelta(\cM_{\widetilde{\cB/q}^{(N)}})=0$.
Conversely, if $\cB$ is not taut, then there are $r\in\N$ and a Behrend set $\cA$ such that $r\cA\subseteq\cB$. In particular $\cA\subseteq\cB/r$, so that also $\cB/r$ is Behrend. But then also all sets $\widetilde{\cB/r}^{(N)}$ are Behrend in view of Theorem~\ref{theo:behrend}, so that
$\lim_{N\to\infty}\ddelta(\cM_{\widetilde{\cB/r}^{(N)}})=1\neq0$.
\end{proof}

\subsection{Consequences for the dynamics of $\cB$-free systems}

For a given set $\cB\subseteq\N$ denote by $\eta\in\{0,1\}^\Z$ the characteristic function of $\cF_\cB$, i.e. $\eta(n)=1$ if and only if $n\in\cF_\cB$, and consider the
orbit closure $X_\eta$ of $\eta$ in the shift dynamical system
$(\{0,1\}^\Z,\sigma)$, where $\sigma$ stands for the left shift.
Topological dynamics and ergodic theory provide a wealth of concepts to describe various aspects of the structure of $\eta$, see \cite{Sa} which originated
 this point of view by studying the set of square-free numbers, and also
 \cite{Ab-Le-Ru}, \cite{BKKL2015}, \cite{KKL2016}, \cite{KR2015}  which continued this line of research. We collect some facts from these references:
\begin{enumerate}[(A)]
\item $\eta$ is quasi-generic for a natural ergodic shift invariant probability measure $\nu_\eta$ on $\{0,1\}^\Z$, called the \emph{Mirsky measure} of $\cB$ \cite[Prop.~E]{BKKL2015},
in particular $\supp(\nu_\eta)\subseteq X_\eta$. The Mirsky measure can 
be characterized 
as the unique shift invariant probability measure $P$ on $X_\eta\subseteq\{0,1\}^\Z$
with the property that $\lim_{n\to\infty}n^{-1}\sum_{k=1}^nx_k=\overline{d}(\cF_\cB)$ for $P$-a.a.~$x\in X_\eta$ (while $\limsup_{n\to\infty}n^{-1}\sum_{k=1}^nx_k\leqslant \overline{d}(\cF_\cB)$ for all $x\in X_\eta$), see \cite[Cor.~3 and~4]{KR2015}. 
\item If $\cB$ has light tails, then $\cB$ is taut, but the converse does not hold \cite[Sect.~4.3 and Cor.~4.19]{BKKL2015}.
\item If $\cB$ has light tails, then $\eta$ is generic for $\nu_\eta$ \cite[Prop.~E and Rem.~2.24]{BKKL2015}.\label{item:LT-generic}
\item If $\cB$ has light tails, then $\supp(\nu_\eta)=X_\eta$ \cite[Thm.~G]{BKKL2015}.\label{item:LT-support}
\item If $\cB$ has light tails and if $\cB$ contains an infinite pairwise coprime subset, then $X_\eta$ is \emph{hereditary}, i.e. $y\in\{0,1\}^\Z$ belongs to $X_\eta$ whenever there is $x\in X_\eta$ with $y\leqslant x$ coordinate-wise \cite[Thm.~D]{BKKL2015}.\label{item:LT-hereditary}
\end{enumerate}
One may ask, whether implications (\ref{item:LT-generic}) - (\ref{item:LT-hereditary}) continue to hold if only tautness of the set $\cB$ is assumed. For Implication (\ref{item:LT-generic}) this is not true \cite[Prop.~4.17]{BKKL2015}, but for the other two implications this remained open in \cite{BKKL2015}. 
Here we prove that it suffices indeed to assume tautness for the conclusions of (\ref{item:LT-support}) and~(\ref{item:LT-hereditary}) to hold true:
\begin{theorem}\label{theo:support}
Suppose that the  set $\cB\subseteq\N$ is taut. Then
$\supp(\nu_\eta)=X_\eta$.
\end{theorem}
\begin{theorem}\label{theo:heredity}
Suppose that the  set $\cB\subseteq\N$ is taut and contains an infinite co-prime subset. Then
 $X_\eta$ is hereditary.
\end{theorem}
\begin{remark}
It was proved in \cite[Thm.~B]{BKKL2015} that $\cB$ contains an infinite co-prime subset if and only if the subshift $X_\eta$ is \emph{proximal}, i.e. if and only if it has a fixed point as its unique minimal subset (the point $(\dots000\dots)$ in this case).
\end{remark}
The proofs of both theorems rely on substantial parts of the proofs of the corresponding results from~\cite{BKKL2015}. We strengthen some of the lemmas from that paper in such a way that light tails are no longer needed to conclude, but the new characterization of tautness from Corollary~\ref{coro:taut} suffices.

Theorem~\ref{theo:behrend} is a $0$-$1$-law that we prove in a measure theoretic and probabilistic framework, which is borrowed from previous publications \cite{BKKL2015,KKL2016,KR2015,KR2016}:
\begin{itemize}
\item $\Delta:\Z\to\prod_{b\in\cB}\Z/b\Z$, $\Delta(n)=(n,n,\dots)$, denotes the canonical diagonal embedding.
\item $H:=\overline{\Delta(\Z)}$ is a compact abelian group, and we denote by $m_H$ its normalised Haar measure.
\item The \emph{window} associated to $\cB$  is defined as
\begin{equation}\label{eq:W}
W:=\{h\in H: h_b\neq0\ (\forall b\in\cB)\}.
\end{equation}
\item For an arbitrary subset $A\subseteq H$ we define the coding function
$\varphi_A:H\to\{0,1\}^\Z$ by $\varphi_A(h)(n)=1$ if  and only if $h+\Delta(n)\in A$. Of particular interest is the coding functions $\varphi:=\varphi_W$
\item
Observe that
$\varphi(h)(n)=1$ if  and only if  $h_b+n\neq0$ mod $b$ for all $b\in\cB$.
\item With this notation $\eta=\varphi(\Delta(0))$ and $X_\eta=\overline{\varphi(\Delta(\Z))}$, so that
$X_\eta\subseteq X_\varphi:=\overline{\varphi(H)}$.
\end{itemize}
Our proof yields indeed the following sharpening of Theorem~\ref{theo:support}:
\begin{theorem}\label{theo:main}
Suppose that the set $\cB\subseteq\N$ is taut. Then
$\supp(\nu_\eta)=X_\eta=X_\varphi$.
\end{theorem}
\noindent
In \cite[Prop.~2.2]{KKL2016} (the second part of) this conclusion was proved under the assumption that $\cB$ has light tails.
\begin{remark}
\begin{compactenum}[a)]
\item
We recall from \cite[Theorem~A]{KKL2016} a purely measure theoretic characterization of tautness: The primitive set $\cB$ is taut if and only if
the window $W$ associated to $\cB$ is Haar-regular, i.e. if $\supp(m_H|_W)=W$.
\item Also proximality of $X_\eta$ (which is equivalent to $\cB$ having no infinite co-prime subset) can be characterized in terms of the window: $X_\eta$ is proximal if and only if $W$ has no interior point \cite[Thm.~C]{KKL2016}.
\item In \cite[Cor.~1]{Keller-heredity}, heredity of $X_\varphi$ is proved under the sole assumptions that $\cB$ is primitive and contains an infinite co-prime subset - no tautness is assumed. (Assuming tautness, heredity of $X_\varphi$ would follow immediately from Theorems~\ref{theo:heredity} and~\ref{theo:main}.) {\it Added in proof:} As the referee remarked, heredity of $X_\eta$ would not follow from these weaker assumptions: Take $\cB$ to be the set of all prime numbers. Then $\eta=\dots000101000\dots$ and the block $00100$ does not appear on $\eta$.
\end{compactenum}
\end{remark}

\paragraph{Acknowledgement} The approach taken in this note occured while I was supervising the MSc thesis of Jakob Seifert \cite{seifert}, who proved the identity $\supp(\nu_\eta)=X_\eta$ under an assumption on the set $\cB$ which implies tautness and is strictly weaker than light tails, but does not seem to be equivalent to tautness, namely: for any finite set $A\subseteq\cP$ there is a thin set $P\subseteq\cP\setminus A$ such that the set $\cB\setminus\cM_P$ has light tails. ($P$ is thin if $\sum_{p\in P}\frac{1}{p}$ converges.)

\section{Proof of Theorem~\ref{theo:behrend}}\label{sec:behrend}

For any subset $\cB'\subseteq\cB$ we denote the corresponding objects defined as above by $\Delta',H',m_{H'},W'$ and $\varphi'$. On the other hand one can consider the window corresponding to $\cB'$ as a subset of $H$, namely $W_{\cB'}:=\{h\in H: h_b\neq0\ (b\in\cB')\}$.

\begin{lemma}\label{lemma:windows}
With the previous notation, $m_H(W_{\cB'})=m_{H'}(W')$.
\end{lemma}
\begin{proof}
Denote by $\pi$ the natural projection from $\prod_{b\in\cB}\Z/b\Z$ to $\prod_{b\in\cB'}\Z/b\Z$. Then $\Delta'(\Z)=\pi(\Delta(\Z))$, and as $\pi$ is continuous between compact metric spaces, it follows that $\pi(H)=H'$ so that $m_{H'}=m_H\circ\pi^{-1}$.
Quite obviously, $\pi^{-1}(W')\subseteq W_{\cB'}$. 
For the converse inclusion let $h\in W_{\cB'}\subseteq H$. Then $(\pi(h))_b=h_b\neq0$ for all $b\in\cB'$ so that $\pi(h)\in W'$. 
Hence
$
m_H(W_{\cB'})
=
m_H(\pi^{-1}(W'))
=
m_{H'}(W')
$.
\end{proof}

\begin{lemma}\label{lemma:dichotomy}
Let $\cB\subseteq\N$ be primitive. Then either $\ddelta(\cM_{\widetilde{\cB}^{(N)}})=1$ for all $N\in\N$ or 
$\lim_{N\to\infty}\ddelta(\cM_{\widetilde{\cB}^{(N)}})=0$.
\end{lemma}
\begin{proof}
In \cite[Lemma~4.1]{KKL2016} it was proved that $m_H(W)=1-\dl(\cM_\cB)$ and, analogously,
$m_{H'}(W')=1-\dl(\cM_{\cB'})$ for each $\cB'=\widetilde{\cB}^{(N)}$. Hence 
$\ddelta(\cM_\cB)=\dl(\cM_\cB)=1-m_H(W)$, and
Lemma~\ref{lemma:windows} implies
\begin{equation*}
\ddelta(\cM_{\cB'})=\dl(\cM_{\cB'})=1-m_{H'}(W')=1-m_H(W_{\cB'})\quad\text{for all }\cB'=\widetilde{\cB}^{(N)}.
\end{equation*}
Observing that $(W_{\widetilde\cB^{(N)}})_N$ is an increasing sequence of sets and denoting
$W_\infty:=\bigcup_{N\in\N}W_{\widetilde{\cB}^{(N)}}$, we thus conclude that
\begin{equation*}
\lim_{N\to\infty}\ddelta(\cM_{\widetilde{\cB}^{(N)}})
=
1-\lim_{N\to\infty}m_H(W_{\widetilde{\cB}^{(N)}})
=
1-m_H\left(W_\infty\right),
\end{equation*}
and, in order to prove the lemma, we must show that either $m_H(W_{\widetilde{\cB}^{(N)}})=0$ for all $N\in\N$, or $m_H\left(W_\infty\right)=1$. This will result from a variant of Kolmogorov's $0$-$1$-law.

For $b\in \cB$ define the random variable $Z_b:H\to\Z$ by $Z_b(h)=h_b$. If $\cA\subseteq\cB$ and $\cC\subseteq\cB$ are co-prime to each other, i.e. if $\gcd(a,c)=1$ for all $a\in\cA$ and all $c\in\cC$, then the families $(Z_a)_{a\in\cA}$ and $(Z_c)_{c\in\cC}$ are independent from each other.
This is a consequence of the generalized Chinese Remainder Theorem, which guarantees that each cylinder set determined by a finite index set $\cB'$ has Haar measure $1/\lcm(\cB')$.

For $\cB'\subseteq\cB$ denote by $\Pi_{\cB'}$ the $\sigma$-algebra generated by the random variables $Z_b$ $(b\in\cB')$. Then $W_\infty\in\Pi_{\widetilde{\cB}^{(N)}}$ for all $N\in\N$, because
$W_\infty=\bigcup_{N'\geqslant N}W_{\widetilde{\cB}^{(N')}}$ and
$W_{\widetilde{\cB}^{(N')}}\in\Pi_{\widetilde{\cB}^{(N')}}\subseteq\Pi_{\widetilde{\cB}^{(N)}}$ whenever $N'\geqslant N$.

Let $\varepsilon>0$. 
\begin{compactenum}[-]
\item As $\Pi_\cB$ is generated by the algebras $\Pi_{\cB^{(N)}}$ $(N\in\N)$
where $\cB^{(N)}:=\{b\in\cB:\Spec(b)\subseteq\{1,\dots,N\}\}$,
there are $N_1\in\N$ and a set
$V_1\in\Pi_{\cB^{(N_1)}}$ such that
$m_H(W_\infty\triangle V_1)<\varepsilon$. Note that all $\cB^{(N)}$ are finite, because $\cB$ is primitive \cite[Lemma~5.14]{BKKL2015}.
\item As $\Pi_{{\widetilde\cB}^{(N_1)}}$ is generated by the algebras 
$\Pi_{\cB'}$ ($\cB'\subseteq{\widetilde\cB}^{(N_1)}$ finite), 
there are a finite set $\cB_2'\subseteq\widetilde\cB^{(N_1)} $ and a set $V_2\in\Pi_{\cB'_2}$ such that
$m_H(W_\infty\triangle V_2)<\varepsilon$.
\item  As $\cB^{(N_1)}$ and $\cB'_2\subseteq{\widetilde\cB}^{(N_1)}$ are co-prime to each other,
the corresponding $\sigma$-algebras $\Pi_{\cB^{(N_1)}}$ and $\Pi_{\cB'_2}$
are independent from each other (see above), in particular $m_H(V_1\cap V_2)=m_H(V_1)\cdot m_H(V_2)$.
\end{compactenum}
 Hence
\begin{equation*}
\left|m_H(W_\infty)-m_H(W_\infty)\cdot m_H(W_\infty)\right|
\leqslant
\left|m_H(V_1\cap V_2)-m_H(V_1)\cdot m_H(V_2)\right|+4\varepsilon
=4\varepsilon\,.
\end{equation*}
As $\varepsilon>0$ was arbitrary, this shows that $m_H(W_\infty)\in\{0,1\}$.
Finally note that if $m_H(W_\infty)=0$, then also $m_H(W_{\widetilde\cB^{(N)}})=0$ for all $N$.
\end{proof}

\begin{proof}[Proof of Theorem~\ref{theo:behrend}]
(i)\; Suppose first that all $\widetilde\cB^{(N)}$ are Behrend. Then $\cB$ is Behrend, because $\widetilde\cB^{(N)}\subseteq\cB$.
If, conversely, there is a non-Behrend set $\widetilde\cB^{(N)}$, then $\cB$ is contained in the finite union $\widetilde\cB^{(N)}\cup\bigcup_{p\in\cP\cap\{1,\dots,N\}}p\cdot\Z$ of non-Behrend sets, and hence $\cB$ is not Behrend \cite[Cor.~0.14]{hall-book}.\\
(ii)\;
This follows from assertion (i) in view of Lemma~\ref{lemma:dichotomy}.
\end{proof}

Stanis\l{}aw Kasjan provided another, purely arithmetic proof of Lemma~\ref{lemma:dichotomy}. I am indebted to him for the permission to reproduce it here
 \cite{Kasjan-pc}:
\begin{proof}[Alternative proof of Lemma~\ref{lemma:dichotomy}]
First observe that if $\cA,\cC\subseteq \N$ are such that $\gcd(a,c)=1$ for every $a\in\cA$, $c\in\cC$, then 
\begin{equation}\label{independent}
\ddelta(\cM_{\cA}\cap \cM_{\cC})=\ddelta(\cM_{\cA})\cdot\ddelta(\cM_{\cC}).
\end{equation}
For finite $\cA$, $\cC$ this is proved
in \cite[Lemma 4.22]{BKKL2015}, the general case is then derived using the Davenport-Erd\"os formula \eqref{eq:app-from-inside}, 
observing that $\cM_\cA\cap\cM_\cC=\cM_{[\cA,\cC]}$ is a set of multiples, where $[\cA,\cC]:=\{\lcm(a,c):a\in\cA,c\in\cC\}$.

Assume now that $\lim_{N\rightarrow\infty}\ddelta(\cM_{\widetilde{\cB}^{(N)}})\neq 0$. Then 
\begin{equation}\label{below}
\ddelta(\cM_{\widetilde{\cB}^{(N)}})\ge \varepsilon
\end{equation}
 for every $N$ and some $\varepsilon>0$.
Note that by \eqref{eq:app-from-inside},
\begin{equation*}\label{lim1}
\lim_{N\rightarrow\infty}\ddelta(\cM_{\widetilde{\cB}^{(N)}}\setminus \cM_{\cB^{(N)}})
\leqslant\lim_{N\rightarrow\infty}\ddelta(\cM_{\cB}\setminus \cM_{\cB^{(N)}})= 0,
\end{equation*} 
and by \eqref{independent},
$$
\ddelta(\cM_{\widetilde{\cB}^{(N)}})=
\ddelta(\cM_{\widetilde{\cB}^{(N)}}\setminus \cM_{\cB^{(N)}})
+ \ddelta(\cM_{\widetilde{\cB}^{(N)}})\cdot\ddelta(\cM_{\cB^{(N)}}).
$$
Hence
\begin{equation*}
\lim_{N\to\infty}\ddelta(\cM_{\widetilde{\cB}^{(N)}})
\cdot\big(1-\ddelta(\cM_{\cB^{(N)}})\big)
=
\lim_{N\to\infty}\ddelta(\cM_{\widetilde{\cB}^{(N)}}\setminus \cM_{\cB^{(N)}})
=0.
\end{equation*}
Together with (\ref{below}) this yields
$$
\lim_{N\rightarrow\infty}(1-\delta(\cM_{\cB^{(N)}}))=0.
$$
Invoking Eq.~\eqref{eq:app-from-inside} once more,
this implies
$\delta(\cM_{\cB})=1$.  
Finally
$\delta(\cM_{\widetilde\cB^{(N)}})=1$ for every $N$ follows from Theorem~\ref{theo:behrend}(i), 
the simple proof of which is of purely number theoretic nature\footnote{The main ingredient of the proof, \cite[Cor.~0.14]{hall-book}, is a direct consequence of Behrend's inequality \cite{Behrend1948}, see also \cite{RT1996}.} and does not rely on Lemma~\ref{lemma:dichotomy}.
\end{proof}
\begin{remark}
In the present context of $\cB$-free dynamics the purely number theoretic proof is certainly the more direct (and hence preferable) one. 
Having in mind that the sets $\cF_\cB$ are very special example of model sets (see e.g.~\cite[Sec.~3.3]{KR2016} for a detailed discussion), the probabilistic proof might indicate how to use 
$0$-$1$-laws for the investigation of  more geometrically defined model sets.
\end{remark}

\section{Proof of Theorems~\ref{theo:support},~\ref{theo:heredity} and~\ref{theo:main}}
Denote by $\cP$ the set of prime numbers. Recall that $\cB^{(n)}:=\{b\in\cB:\Spec(b)\subseteq\{1,\dots,n\}\}$. For a finite set $A\subseteq\cP$ denote
$$\cB^{(A)}:=\{b\in\cB:\Spec(b)\subseteq A\}.$$
\begin{lemma}\label{lemma:prime-exhaust}
Suppose that 
the primitive set $\cB$ is taut.
Then for each finite set $A\subseteq\cP$ and each $\varepsilon>0$ there is a finite set $P\subseteq\cP$ such that
\begin{equation}
P\cap A=\emptyset\quad\text{and}\quad\ddelta(\cM_{(\cB\setminus\cB^{(A)})\setminus\cM_P})<\varepsilon.
\end{equation}
\end{lemma}
\begin{proof}
Denote $a:=\card A$ and $K:=\sum_{p\in A}\frac{1}{p}$. Choose $L\in\N$ large enough that $\sum_{p\in A}\frac{1}{p^L}<\varepsilon/2$ and let
\begin{equation*}
Q:=\left\{\prod_{p\in A}p^{k_p}:k_p\in\N_0\right\}
\quad\text{and}\quad
Q_0:=\left\{\prod_{p\in A}p^{k_p}\in Q: k_p<L\ (p\in A)\right\}.
\end{equation*}
In view of Corollary~\ref{coro:taut}, we can fix $N\in\N$ large enough that $\ddelta(\cM_{\widetilde{\cB/q}^{(N)}})<\varepsilon/(2L^a)$ for all $q\in Q_0$.

Let $P:=\big(\cP\cap\{1,\dots,N\}\big)\setminus A$. Then
\begin{equation*}
\begin{split}
\cB\setminus(\cB^{(A)}\cup\cM_P)
&\subseteq
\bigcup_{q\in Q}q\cdot\widetilde{\cB/q}^{(N)}
\subseteq
\bigcup_{q\in Q_0}q\cdot\widetilde{\cB/q}^{(N)}
\cup\bigcup_{q\in Q\setminus Q_0}q\cdot\Z\\
&\subseteq
\bigcup_{q\in Q_0}q\cdot\widetilde{\cB/q}^{(N)}
\cup\bigcup_{p\in A}p^L\cdot\Z\,,
\end{split}
\end{equation*}
so that
\begin{equation*}
\cM_{\cB\setminus(\cB^{(A)}\cup\cM_P)}
\subseteq
\bigcup_{q\in Q_0}q\cdot\cM_{\widetilde{\cB/q}^{(N)}}
\cup\bigcup_{p\in A}p^L\cdot\Z\,.
\end{equation*}
Hence
\begin{equation*}
\ddelta(\cM_{\cB\setminus(\cB^{(A)}\cup\cM_P)})
\leqslant
\sum_{q\in Q_0}\frac{1}{q}\,\ddelta(\cM_{\widetilde{\cB/q}^{(N)}})+\sum_{p\in A}\frac{1}{p^L}
\leqslant
\card Q_0\cdot\frac{\varepsilon}{2L^a}+\frac\varepsilon2
=
\varepsilon\,.
\end{equation*}
\end{proof}

Next we prove a strengthening of Lemma~5.20 from \cite{BKKL2015}.
\begin{lemma}\label{lemma:as5.20}
Let $\beta,r,n\in\N$ and $\cC\subseteq\N$. Assume that $P\subseteq\{n+1,n+2,\dots\}$ is a finite set of prime numbers co-prime to $\beta$. Then the logarithmic density $\ddelta\left((\beta\Z+r)\cap\bigcap_{i=1}^n(\cF_\cC-i)\right)$ exists
and
\begin{equation}\label{eq:new-label}
\ddelta\left((\beta\Z+r)\cap\bigcap_{i=1}^n(\cF_\cC-i)\right)
\geqslant
\prod_{p\in P}\left(1-\frac{n}{p}\right)\cdot 
\ddelta\left((\beta\Z+r)\cap\bigcap_{i=1}^n(\cF_{\cC\setminus\cM_P}-i)\right).
\end{equation}
\end{lemma}
\begin{proof}
We start with the existence of the logarithmic density. Denote by $\uddelta$ and $\oddelta$ the lower and upper logarithmic density, respectively. 
{Both are obviously monotone, invariant under shifting the sequence, and $\oddelta$ is finitely subadditive.}
\red{For any sets $U,V\subseteq\N$ they satisfy $\uddelta(U\cup V)\leqslant\uddelta(U)+\oddelta(V)$, and if $U$ and $V$ are disjoint, then also $\uddelta(U)+\oddelta(V)\leqslant\oddelta(U\cup V)$. (These are elementary consequences of the ``sum rule'' for $\limsup$ and $\liminf$.)}
 Hence, for $M\in\N$,
\begin{equation*}
\begin{split}
0\leqslant&
\ddelta\left((\beta\Z+r)\cap\bigcap_{i=1}^n(\cF_{\cC\cap\{1,\dots,M\}}-i)\right)
-
\uddelta\left((\beta\Z+r)\cap\bigcap_{i=1}^n(\cF_{\cC}-i)\right)\\
=&
\red{\uddelta}\left(\beta\Z\setminus\bigcup_{i=1}^{n}(\cM_{\cC\cap\{1,\dots,M\}}-i-r)\right)
-
\uddelta\left(\beta\Z\setminus\bigcup_{i=1}^{n}(\cM_{\cC}-i-r)\right)\\
\leqslant&
\oddelta\left(\left(\bigcup_{i=1}^{n}(\cM_{\cC}-i-r)\right)\Big\backslash
\left(\bigcup_{i=1}^{n}(\cM_{\cC\cap\{1,\dots,M\}}-i-r)\right)\right)\\
\leqslant&
\sum_{i=1}^n\oddelta\left((\cM_{\cC}-i-r)\setminus
(\cM_{\cC\cap\{1,\dots,M\}}-i-r)\right)\\
\leqslant&
\sum_{i=1}^n\Big(\oddelta\left(\cM_{\cC}-i-r\right)
-\uddelta\left(\cM_{\cC\cap\{1,\dots,M\}}-i-r\right)\Big)\\
=&
\sum_{i=1}^n\Big(\ddelta\left(\cM_{\cC}\right)
-\ddelta\left(\cM_{\cC\cap\{1,\dots,M\}}\right)\Big),
\end{split}
\end{equation*}
\red{}
It follows that
\begin{equation*}
\begin{split}
\oddelta\left((\beta\Z+r)\cap\bigcap_{i=1}^n(\cF_{\cC}-i)\right)
&\leqslant
\oddelta\left((\beta\Z+r)\cap\bigcap_{i=1}^n(\cF_{\cC\cap\{1,\dots,M\}}-i)\right)
=
\ddelta\left((\beta\Z+r)\cap\bigcap_{i=1}^n(\cF_{\cC\cap\{1,\dots,M\}}-i)\right)\\
&\leqslant
\uddelta\left((\beta\Z+r)\cap\bigcap_{i=1}^n(\cF_{\cC}-i)\right)
+
{n\cdot}\Big(\ddelta\left(\cM_{\cC}\right)
-\ddelta\left(\cM_{\cC\cap\{1,\dots,M\}}\right)\Big).
\end{split}
\end{equation*}
Observing equation \eqref{eq:app-from-inside}, this yields in the limit $M\to\infty$ that
\begin{equation*}
\ddelta\left((\beta\Z+r)\cap\bigcap_{i=1}^n(\cF_{\cC}-i)\right)
=
\lim_{M\to\infty}\ddelta\left((\beta\Z+r)\cap\bigcap_{i=1}^n(\cF_{\cC\cap\{1,\dots,M\}}-i)\right).
\end{equation*}

We turn to the proof of inequality \eqref{eq:new-label}.
From \cite[Lemma 5.18]{BKKL2015} we know that
\begin{equation}\label{eq:approximation-new}
d\left((\beta\Z+r)\cap\bigcap_{i=1}^n(\cF_{\cC\cap\{1,\dots,M\}}-i)\right)
\geqslant
\left(1-\frac{n}{p}\right)\cdot 
d\left((\beta\Z+r)\cap\bigcap_{i=1}^n(\cF_{(\cC\setminus p\Z)\cap\{1,\dots,M\}}-i)\right)
\end{equation}
for each $p\in P$. Applying this inductively to all $p\in P$ (replacing $\cC$ by $\cC\setminus p\Z$ etc.), this yields
\begin{equation*}
d\left((\beta\Z+r)\cap\bigcap_{i=1}^n(\cF_{\cC\cap\{1,\dots,M\}}-i)\right)
\geqslant
\prod_{p\in P}\left(1-\frac{n}{p}\right)\cdot 
d\left((\beta\Z+r)\cap\bigcap_{i=1}^n(\cF_{(\cC\setminus \cM_P)\cap\{1,\dots,M\}}-i)\right),
\end{equation*}
and the same holds, of course, for the logarithmic density $\ddelta$. As the (logarithmic) density is monotone, we obtain
\begin{equation*}
\ddelta\left((\beta\Z+r)\cap\bigcap_{i=1}^n(\cF_{\cC\cap\{1,\dots,M\}}-i)\right)
\geqslant
\prod_{p\in P}\left(1-\frac{n}{p}\right)\cdot 
\ddelta\left((\beta\Z+r)\cap\bigcap_{i=1}^n(\cF_{(\cC\setminus \cM_P)}-i)\right)
\end{equation*}
for all $M\in\N$. An application of \eqref{eq:approximation-new} finishes the proof of the lemma.
\end{proof}

\begin{proposition}
\label{prop:to5.20}
Let $\beta,r,n\in\N$, assume that the primitive set $\cB\subseteq\N$ is taut and denote
$A:=\Spec(\beta)\cup(\cP\cap\{1,\dots,n\})$. Then
\begin{equation}\label{eq:to5.20}
\ddelta\left((\beta\Z+r)\cap\bigcap_{i=1}^n(\cF_{\cB\setminus\cB^{(A)}}-i)\right)>0\,.
\end{equation}
\end{proposition}
\begin{proof}
Apply Lemma~\ref{lemma:prime-exhaust} with 
 $\varepsilon:=\frac{1}{2n\beta}$. This produces a finite set $P\subseteq\cP\setminus A$, hence co-prime to $\beta$, with $\ddelta(\cM_{\cB\setminus(\cB^{(A)}\cup\cM_P)})<\varepsilon$. Hence
\begin{equation*}
\ddelta\left((\beta\Z+r)\Big\backslash\bigcap_{i=1}^n(\cF_{\cB\setminus(\cB^{(A)}\cup\cM_P)}-i)\right)
\leqslant
\ddelta\left(\bigcup_{i=1}^n(\cM_{\cB\setminus(\cB^{(A)}\cup\cM_P)}-i)\right)
\leqslant
\sum_{i=1}^n\ddelta\left(\cM_{\cB\setminus(\cB^{(A)}\cup\cM_P)}-i\right)
<n\varepsilon=\frac{1}{2\beta}.
\end{equation*}
Combining this with Lemma~\ref{lemma:as5.20} (applied with $\cC=\cB\setminus\cB^{(A)}$) yields
\begin{equation*}
\begin{split}
\ddelta\left((\beta\Z+r)\cap\bigcap_{i=1}^n(\cF_{\cB\setminus\cB^{(A)}}-i)\right)
&\geqslant
\prod_{p\in P}\left(1-\frac{n}{p}\right)\cdot 
\ddelta\left((\beta\Z+r)\cap\bigcap_{i=1}^n(\cF_{\cB\setminus(\cB^{(A)}\cup\cM_P)}-i)\right)\\
&=
\prod_{p\in P}\left(1-\frac{n}{p}\right)\cdot \left(\ddelta(\beta\Z+r)-\ddelta\left((\beta\Z+r)\Big\backslash\bigcap_{i=1}^n(\cF_{\cB\setminus(\cB^{(A)}\cup\cM_P)}-i)\right)
\right)\\
&\geqslant
\frac{1}{2\beta}\cdot\prod_{p\in P}\left(1-\frac{n}{p}\right)>0\,.
\end{split}
\end{equation*}
\end{proof}

Next we turn to Proposition~5.11 of \cite{BKKL2015} and provide a proof of the same assertion under the sole assumption that the  set $\cB$ is taut. 
\begin{proposition}\label{prop:as5.11}
Assume that the primitive set $\cB$ is taut and that $\cB^{(n)}\subseteq\cA\subseteq\cB$ for some $n>0$.
Suppose that 
\begin{equation}\label{eq:5.11-ass}
\{r+1,\dots,r+n\}\cap\cM_\cA=r+I\text{ for some }r\in\N\text{ and some set }I\subseteq\{1,\dots,n\}.
\end{equation}
Then the density of the set of all $k\in\N$ for which
\begin{equation*}
\{k+1,\dots,k+n\}\cap\cM_\cB=k+I
\end{equation*}
is strictly positive.
\end{proposition}
\begin{proof} 
The proof is strongly inspired by the proof of Proposition~5.11 in \cite{BKKL2015}: For $u\in I$ we choose $b_u\in\cB$ such that $b_{u}\mid r+u$. Without loss of generality we may assume that $\cA=\{{b_{u}}:u\in I\}\cup\cB^{(n)}$. Then, by {\cite[Lemma~5.14]{BKKL2015}}, $\cA$ is finite, and we set $\beta:=\lcm(\cA)$. 

By definition of the set $\cA$, we have for all $i\in\{1,\dots,n\}$
\begin{equation*}
i\in I
\Leftrightarrow
r+i\in\cM_\cA
{\Rightarrow}
{b_{i}}\mid r+i\,.
\end{equation*}
Let $i\in\{1,\dots,n\}$.

If $i\in I$, then ${b_{i}}\mid r+i$, i.e. $r+i\in {b_{i}}\Z$. As ${b_{i}}\mid\lcm(\cA)=\beta$, it follows that $r+\beta\Z+i\subseteq {b_{i}}\Z\subseteq\cM_\cB$. 
Hence
\begin{equation}\label{eq:reduction-1}
\begin{split}
&\Big\{k\in r+\beta\Z:\{k+1,\dots,k+n\}\cap\cM_\cB=k+I\Big\}\\
=&
\Big\{k\in r+\beta\Z: k\in\bigcap_{i\in I}(\cM_\cB-i)\cap
\bigcap_{i\in\{1,\dots,n\}\setminus I}(\cF_\cB-i)\Big\}\\
=&
\Big\{k\in r+\beta\Z: k\in
\bigcap_{i\in\{1,\dots,n\}\setminus I}(\cF_\cB-i)\Big\}\\
=&
(r+\beta\Z)\cap\bigcap_{i\in\{1,\dots,n\}\setminus I}(\cF_{\cB}-i).
\end{split}
\end{equation}

Denote {$A:=\Spec(\beta)\cup(\cP\cap\{1,\dots,n\})$} and notice that
\begin{equation*}
\cB^{(A)}=\big\{b\in\cB:\Spec(b)\subseteq\Spec(\beta){\cup\{1,\dots,n\}}\big\}
\end{equation*}
is finite {\cite[Lemma 5.14]{BKKL2015}}. 
As $\Spec(\cB^{(n)})\subseteq\Spec(\cA)=\Spec(\beta)$, we have $\cB^{(n)}\subseteq\cB^{(A)}$.
Let $b\in \cB^{(A)}\setminus\cB^{(n)}$. As $b\not\in\cB^{(n)}$, 
{$b$ has a prime divisor $p>n$, and as $b\in\cB^{(A)}$, this $p$ divides $\beta$, whence $p\mid {b_{u}}$
for some $u\in I$.}
It follows that if
$b\mid r+\beta \ell+i$ for some $1\leqslant i\leqslant n$ and $\ell\in\Z$, then 
$p\mid\big((r+\beta\ell+i)-\beta\ell-(r+u)\big)$, because $p\mid {b_{u}}\mid r+u$. This implies $p\mid i-u$, 
so that $i=u\in I$, because $p>n>|i-u|$. Thus we have shown that 
if $b\mid r+\beta \ell+i$ for some $b\in \cB^{(A)}\setminus\cB^{(n)}$, {$\ell\in\Z$ and $1\leqslant i\leqslant n$}, then $i\in I$.
Equivalently, if $i\in\{1,\dots,n\}\setminus I$, then $r+\beta\Z+i\subseteq\cF_{\cB^{(A)}\setminus\cB^{(n)}}$.
Hence, $r+\beta\Z\subseteq\bigcap_{i\in\{1,\dots,n\}\setminus I}(\cF_{\cB^{(A)}\setminus\cB^{(n)}}-i)$,  
and we can continue the chain of identities from \eqref{eq:reduction-1} by
\begin{equation}
\begin{split}
=&
(r+\beta\Z)\cap\bigcap_{i\in\{1,\dots,n\}\setminus I}(\cF_{\cB^{(n)}}-i)
\cap\bigcap_{i\in\{1,\dots,n\}\setminus I}(\cF_{\cB^{(A)}\setminus\cB^{(n)}}-i)\cap
\bigcap_{i\in\{1,\dots,n\}\setminus I}(\cF_{\cB\setminus \cB^{(A)}}-i)\\
=&
(r+\beta\Z)\cap\bigcap_{i\in\{1,\dots,n\}\setminus I}(\cF_{\cB^{(n)}}-i)\cap\bigcap_{i\in\{1,\dots,n\}\setminus I}(\cF_{\cB\setminus \cB^{(A)}}-i).
\end{split}
\end{equation}
Finally, if $r+\beta\ell+i\in\cM_{\cB^{(n)}}$ for some $\ell\in\Z$, then 
there is $b\in\cB^{(n)}\subseteq\cA$ such that $b\mid r+\beta\ell+i$ and $b\mid\beta$.
Hence $b\mid r+i$, so that $i\in I$. Equivalently, if $i\in\{1,\dots,n\}\setminus I$, then $r+\beta\Z+i\subseteq\cF_{\cB^{(n)}}$, and we can finish the above identities by
\begin{equation}
\begin{split}
=&
(r+\beta\Z)\cap\bigcap_{i\in\{1,\dots,n\}\setminus I}(\cF_{\cB\setminus \cB^{(A)}}-i)
\supseteq 
(r+\beta\Z)\cap\bigcap_{i\in\{1,\dots,n\}}(\cF_{\cB\setminus \cB^{(A)}}-i).
\end{split}
\end{equation}
In view of Proposition~\ref{prop:to5.20}, the logarithmic density of the latter set is strictly positive. This finishes the proof of the proposition.
\end{proof}
\begin{proof}[Proof of Theorem~\ref{theo:support}]
We must show that $X_\eta\subseteq\supp(\nu_\eta)$ or, equivalently, that each block $(\eta_{r+1},\dots,\eta_{r+n})$ occurs in $\eta$ with strictly positive frequency (observe that $\eta$ is quasi-generic for $\nu_\eta$). But this is just a rewording of Proposition~\ref{prop:as5.11}.
\end{proof}

\begin{proof}[Proof of Theorem~\ref{theo:heredity}]
The heredity of $X_\eta$ was proved in \cite[sec.~5]{BKKL2015} under the additional assumption that $\cB$ has light tails. This assumption enters the proof only via Proposition~5.11 of that reference, so replacing it by our Proposition~\ref{prop:as5.11} leads to the heredity of $X_\eta$ under the present assumptions.
\end{proof}

\begin{proof}[Proof of Theorem~\ref{theo:main}]
The identity $X_\varphi=X_\eta$ was proved in \cite[Prop.~2.2]{KKL2016} under the assumption that $\cB$ has light tails. Again, this assumption entered only 
via a reference to Proposition~5.11 from \cite{BKKL2015}, which, once more, can be replaced by the present Proposition~\ref{prop:as5.11}. The identity $\supp(\nu_\eta)=X_\eta$ was proved in Theorem~\ref{theo:support}.
\end{proof}

\renewcommand{\em}{\it}

\end{document}